%% file: main.tex
\documentclass{article}

\usepackage[english]{babel}
\usepackage{color}

\usepackage[letterpaper,top=2cm,bottom=2cm,left=3cm,right=3cm,marginparwidth=1.75cm]{geometry}

\input{packages}

\title{Small-Disturbance Input-to-State Stability of Perturbed Gradient Flows: Applications to LQR Problem \thanks{This work has been supported in part by the NSF grants CNS-2227153 and ECCS-2210320 (L.C. and Z.P.J.) and AFOSR grant FA9550-21-1-0289 and ONR grant N00014-21-1-2431 (E.D.S.)}}
\author{Leilei Cui\footnote{Department of Electrical and Computer Engineering, New York University, Brooklyn, NY, USA}, Zhong-Ping Jiang\footnotemark[2], and Eduardo D. Sontag\footnote{Department of Electrical and Computer Engineering and Department of BioEngineering, Northeastern University, Boston, MA, USA}}

\begin{document}
\maketitle

\begin{abstract}
This paper studies the effect of perturbations on the gradient flow of a \LC{general nonlinear} programming problem, where the perturbation may arise from inaccurate gradient estimation in the setting of data-driven optimization. Under suitable conditions on the objective function, the perturbed gradient flow is shown to be small-disturbance input-to-state stable (ISS), which implies that, in the presence of a small-enough perturbation, the trajectories of the perturbed gradient flow must eventually enter a small neighborhood of the optimum. This work was motivated by the question of robustness of direct methods for the linear quadratic regulator problem, and specifically the analysis of the effect of perturbations caused by gradient estimation or round-off errors in policy optimization. \LC{We show} small-disturbance ISS for three of the most common optimization algorithms: standard gradient flow, natural gradient flow, and Newton gradient flow. 
\end{abstract}

\section{Introduction}
Gradient-based optimization of loss functions constitutes a key tool in contemporary machine learning. Thus, the theoretical analysis of convergence to the minima of loss functions, in gradient-like iterations and/or in their continuous analogue, gradient-like flows (viewed as the limit of discrete-time gradient descent algorithms with infinitesimally small step size) have attracted considerable attention from both academic and industrial researchers. 
Besides convergence under ideal no-noise situations, a useful optimization algorithm should be capable of finding a near-optimal solution while degrading elegantly in the face of perturbations that might arise from noisy measurements of experimental data, arithmetic rounding errors due to numerical computation, numerically approximating the gradient from data through two-point estimates, discretization error when solving ordinary differential equations, or even early stopping when estimating gradients in a hierarchical learning setup~\cite{ieee_tac_2018_cherukuri_et_al_convexity_saddle_point_dynamics,2021_arxiv_iss_gradient_suttner_dashkovskiy,2020arxiv_bianchin_poveda_dallanese_gradient_iss_switched_systems,Mohammadi2022,Sontag2022}. Especially in the setting of data-driven optimization, the analytical form of the gradient is typically unknown, and consequently the gradient has to be numerically approximated through sampling and experiments, which unavoidably introduces perturbations to the gradient iteration or flow. In computer science foundations of optimization theory, similarly, noisy or error-prone operations such as inexact or stochastic gradient computations have led to the introduction of the concept of ``reproducibility in optimization'' which is concerned conceptually with the same issues~\cite{reproducibility2022}. One could also view adversarial attacks on neural network training as affecting gradient computations in ``backpropagation'' algorithms, and the effect of disturbances in that context has been the subject of recent work~\cite{23cdc_brancodeoliveira_siami_sontag}.
In summary, both the convergence and robustness properties of gradient descent should be theoretically analyzed in the presence of perturbations. Mathematically, gradient flows are more amenable to mathematical analysis than discrete iterations, so they are the main object of study in this paper.

In order to formulate precisely the effect of perturbations on gradient flows, we employ as in~\cite{Sontag2022} the formalism of input-to-state stability (ISS) introduced originally in~\cite{Sontag1989} (see for example~\cite{Sontag2008} for an exposition).

The key to proving ISS for perturbed gradient flows is to verify a \emph{Polyak-\L{}ojasiewicz (PL)} type of condition~\cite{Polyak1963, Lojasiewicz1963, Hamed2016} on the loss function to be optimized, meaning that the gradient of the loss function should not be ``too small'' compared to the loss. Roughly (precise definitions to be given) if we wish to minimize a continuously differentiable (but not necessarily convex) function $\mathcal{J}(z)$ on a domain $\mathcal{Z}$ and if a global minimizer $z^*$ exists, then we would like that, for some continuous function $\kappa: \R_{\geq0} \rightarrow \R_{\geq 0}$ which satisfies that $\kappa(0)=0$ and $\kappa(r)>0$ for all $r>0$ (a ``positive definite'' function), there should hold an estimate of the form
$\norm{\nabla \mathcal{J}(z)} \ge \kappa(\mathcal{J}(z) - \mathcal{J}(z^*))$
valid for all $z$ in $\mathcal{Z}$.
The classical PL condition is often stated in a slightly different ``semi-global'' form, by requiring the existence, for each $r$, of a constant $c_r$ such that 
$\norm{\nabla \mathcal{J}(z)} \ge c_r (\mathcal{J}(z) - \mathcal{J}(z^*))$
for every $z$ in the sublevel set $\{z\st\mathcal{J}(z) \leq r\}$.
For a coercive function $\mathcal{J}$, this is equivalent to the above estimate using positive definite functions.

Sometimes, the PL condition is stated globally, that is, with a constant $c$ which is independent of $r$ (in other words, $\kappa$ can be picked as a linear function), but such a global condition is too strong for many applications, including the one to be pursued here. On the other hand, one could think of stronger forms of the PL condition, weaker than the existence of a linear function but stronger than merely requiring $\kappa$ to be positive definite. One particularly useful strengthening is to ask that $\kappa$ be a function of class ${\cal K}$, that is, that it be a strictly increasing function. For example, one could take $\kappa(r) = \frac{ar}{b+r}$: note that this function saturates, in the sense that it approaches a finite limit as $r\rightarrow+\infty$. This is stronger than asking that $\kappa$ be only positive definite, as illustrated by $\kappa(r) = \frac{ar}{(b+r)^2}$ which is positive definite but is not of class ${\cal K}$.
An even further strengthening would be to ask that $\kappa$ be of class ${\cal K_{\infty}}$, meaning that $\kappa$ does not saturate, $\kappa(r)\rightarrow\infty$ as $r\rightarrow\infty$, as for example when $\kappa$ is linear.
Estimates with $\kappa$ of class ${\cal K_{\infty}}$ lead to ISS estimates for perturbed gradient flows, as discussed in~\cite{Sontag2022}, and a similar proof to that in~\cite{Sontag2022} can be used to show that estimates with $\kappa(r)$ only positive definite lead to the weaker property of ``integral ISS (iISS)'' \cite{MR1629012,angeli_sontag_wang_iISS_TAC00} for perturbed gradient flows.
We may call the intermediate type of PL estimate, in which $\kappa$ is required to be of class ${\cal K}$ (a stronger property than positive definiteness, but not as strong as class ${\cal K}_{\infty}$) a \emph{CJS-PL} (``comparison just saturated'') estimate.
It turns out that CJS-PL estimates are exactly what is required in order to establish ``small-disturbance'' ISS as studied in~\cite{Pang_Jiang_2021,Pang2022}.
To be precise, we will show that when the objective function is coercive (the value of the objective function blows up when the decision variable approaches the boundary of ${\cal Z}$) and the CJS-PL condition holds, the perturbed gradient flow is small-disturbance ISS. This implies that the trajectories of the perturbed gradient flow will eventually enter a small neighborhood of the optimal solution, as long as the perturbation is sufficiently small. In addition, the size of the neighborhood is (in a nonlinear manner) proportional to the magnitude of the perturbation.
In the application that motivated this work, the linear regulator problem (see below), CJS-PL is the correct notion to use, and we believe that this notion might be of more general applicability in optimization problems as well.

We should remark that generalizations of the PL condition, and relations to ISS types of properties, can be found in other recent work. This includes~\cite{2021_poveda_krstic_fixedtime_iss_extremum_seeking}, which studies the
gradient minimization of a function ${\cal J}_q$ on Euclidean space, where the parameter $q$ represents time-varying uncertainty. In that paper, an ISS property is established with respect to the rate of change of the parameter $q$, essentially showing differential ISS (DISS)~\cite{DiffISS_IJRNC03}).  
Extremum-seeking controllers based on gradient flows and an ISS property with respect to disturbances, specifically for an integrator and a kinematic unicycle, are designed and analyzed in~\cite{2021_arxiv_iss_gradient_suttner_dashkovskiy}; in that paper the domain is a closed submanifold of an Euclidean space.
In~\cite{ieee_tac_2018_cherukuri_et_al_convexity_saddle_point_dynamics} one finds results on gradient flows that are ISS with respect to additive errors, but assuming a ``convex-concave'' property for the loss function, and in ~\cite{2020arxiv_bianchin_poveda_dallanese_gradient_iss_switched_systems} the authors solve an output regulation problem for switched linear systems, and show an ISS property for gradient flows with respect to unknown disturbances acting on the plant.

We now turn to the main motivation for this work. Reinforcement learning (RL) is an active research field in which gradient-based optimization plays a pivotal role \cite[Chapter 13]{book_sutton}. In the setting of RL, an agent interacts continuously with an unknown environment, and iteratively optimizes a performance index by collecting data from the environment. By adopting gradient-based optimization methods, various policy optimization (PO) algorithms have been developed, such as actor-critic methods \cite{konda1999actor}, deep deterministic policy gradient \cite{Lillicrap2015}, and trust region policy optimization \cite{Schulman15}. The critical strategy of the policy optimization methods is to parameterize the policy by universal approximations and update the parameters of the policy along the gradient descent direction of the performance index. 

Starting in the early 1960s with the work of Kalman, the linear quadratic regulator (LQR) problem was shown to be theoretically tractable, and has become a widely utilized tool for optimal control and feedback design in engineering applications. 
In the classical approach, the (infinite-horizon) LQR problem relies upon the solution of a Riccati equation. In 1970, Athans and Levine~\cite{1970levine_athans} introduced the idea of a direct gradient descent computation of optimal feedback gains, a procedure which can be interpreted as a form of RL. Thus, the LQR problem offers an ideal benchmark for better understanding policy optimization methods in the RL field, as one can compare solutions to the known optimal solution, and analysis of gradient methods can take advantage of theory developed for LQR. For policy optimization in the LQR problem, the objective function is a cumulative quadratic function of the state and control inputs, the control policy is parameterized as a linear function (feedback) of the state, and the admissible set, consisting of all the stabilizing control gains, is an open subset of an Euclidean space. 
As investigated for example in \cite{bu2020policy,Mohammadi2022,Hu2023Review}, the gradient of the objective function can be computed by using a Lyapunov equation that depends on the system matrices. Nevertheless, if precise system knowledge is unavailable, as in the setting of model-free RL, the gradient has to be numerically approximated through sampling and experiments. For example, by utilizing the approximate dynamic programming technique \cite{books_Bertsekas,book_Powell}, the Lyapunov equation was solved by data-driven methods in \cite{book_Jiang,Book_Lewis2013,tutorial_Jiang}. In \cite{Mohammadi2022,fazel2018global,Li2021}, the gradient is directly calculated by the finite differences method \cite[Section 7.1]{book_Nocedal}, based on the change in function values in response to small perturbations near a given point. For these data-driven methods, a gradient estimation error is inevitable due to noisy data and insufficient samples. Therefore, the robustness analysis of the policy optimization algorithm in the presence of perturbations is critical for efficient learning, and lays the foundations for better understanding RL algorithms.

Our main result will be that, for the LQR problem, the loss function is coercive and satisfies the CJS-PL property, and therefore, by the results in the first part of the paper, we conclude that the perturbed standard gradient flow is small-disturbance ISS. We also show that two variants of gradient flows, natural gradient flows and Newton gradient flows, are small-disturbance ISS. The new contribution is to establish the CJS-PL property for the LQR problem. This considerably extends previous work~\cite{Mohammadi2022,bu2020policy} that only showed a semiglobal estimate (and thus would imply merely iISS). In \cite{Sontag2022}, it was mistakenly stated that the magnitude of the gradient is lower bounded by a $\mathcal{K}_\infty$-function, which is a stronger property. This is incorrect.
Indeed, take a one-dimensional linear system with scalar inputs and assume that all constants in the system and cost function are equal to one. Then the loss function is
${\cal J}(z) = \frac{z^2 + 1}{2(z - 1)}$,
so that its gradient is ${\cal J'}(z) = \frac{z^2 - 2z - 1}{2(z - 1)^2}$. 
The domain of ${\cal J}$ is the open set $(1,\infty)$. We claim that there is no function $\kappa$ of class $\mathcal{K}_\infty$ such that 
$\abs{\mathcal{J}'(z)} \ge \kappa(\mathcal{J}(z) - \mathcal{J}(z^*))$,
where $z^*$ is the global minimizer of $\mathcal{J}$. Indeed, as $z\rightarrow\infty$ we would have that $\mathcal{J}(z)\rightarrow\infty$, so, as $\kappa$ is of class $\mathcal{K}_\infty$, also $\kappa(\mathcal{J}(z) - \mathcal{J}(z^*))\rightarrow\infty$. However, the left-hand side $\abs{\mathcal{J}'(z)}$ is bounded, and in fact converges to $1/2$, showing a contradiction.

To summarize, the contributions of the paper are as follows. First, we provide a Lyapunov-like necessary and sufficient condition for small-disturbance ISS. Second, under assumptions of coercivity and the CJS-PL property, we use the Lyapunov characterization to show that the perturbed gradient flow for a general constrained nonlinear programming problem is small-disturbance ISS. Finally, we show the CJS-PL property for the LQR loss function, which in turn then implies that the standard gradient flow, natural gradient flow, and Newton gradient flow, are all small-disturbance ISS.
The remaining contents of the paper are organized as follows. The notations and preliminaries are introduced in Section 2.  In Section 3, the concept of small-disturbance ISS is reviewed, followed by a necessary and sufficient condition. Section 4 introduces the perturbed gradient flow for a general constrained nonlinear programming problem over an open admissible set, and it is shown that the perturbed gradient flow is small-disturbance ISS under appropriate conditions on the loss function. In Section 5, we study the CSJ-PL property for the LQR problem, and three different kinds of the perturbed gradient flows for LQR are shown to be small-disturbance ISS. Some concluding remarks are given in Section 6.

\section{Notations and Preliminaries}

In this paper, $\mathbb{R}$ ($\mathbb{R}_+$) denotes the set of (nonnegative) real numbers. $\mathbb{P}^n$ denotes the set of $n$-dimensional real symmetric and positive definite matrices. $\eigmin{\cdot}$ and $\eigmax{\cdot}$ denote the minimal and maximal eigenvalues of a real symmetric matrix, respectively. $\Tr{\cdot}$ denotes the trace of a square matrix. $\norm{\cdot}$ denotes the spectral norm of a matrix or Euclidean norm of a vector, and $\norm{\cdot}_F$ denotes the Frobenius norm of a matrix. $\mathcal{L}_\infty^{n}$ ($\mathcal{L}_\infty^{m \times n}$) denotes the set of measurable and locally essentially bounded functions $w: \mathbb{R}_+ \to \Rn$ ($K: \mathbb{R}_+ \to \mathbb{R}^{m \times n}$), endowed with the essential supremum norm  $\norm{w}_\infty = \esssup_{s \in \mathbb{R}_+}\norm{w(s)}$ ($\norm{K}_\infty = \esssup_{s \in \mathbb{R}_+}\norm{K(s)}_F$). $w_t$ denotes the truncation of $w$ at $t$, that is, $w_t(s) = w(s)$ if $s \le t$, and $w_t(s) = 0$ if $s > t$. $I_n$ denotes the $n$-dimensional identity matrix. $\mathrm{Id}$ denotes the identity function. For any $K_1, K_2 \in \mathbb{R}^{m \times n}$ and $Y \in \mathbb{P}^n$, define the inner product $\innprod{K_1}{K_2}_Y = \Tr{K_1^T K_2 Y}$. In addition, to simplify the notation, we denote $\innprod{K_1}{K_2} = \innprod{K_1}{K_2}_{I_n}$. Recall that for any $K \in \mathbb{R}^{m \times n}$, $\norm{K}_F^2 = \innprod{K}{K}$. For any two real symmetric matrices $A$ and $B$, $A \succ B$ \LC{means} that $A-B$ is positive definite, and $A \succeq B$ \LC{means} that $A-B$ is positive semidefinite.

\LC{In the remainder of this section we gather definitions and technical results which we shall use throughout the rest of the paper. A number of these results appear across the literature.}

\begin{definition}[Definitions 2.5 and 24.2 in \cite{book_Hahn}]
A function $\alpha: \mathbb{R}_+ \to \mathbb{R}_+$ is a $\mathcal{K}$-function if it is continuous, strictly increasing, and vanishes at zero. For any $d >0$, a function $\alpha: [0,d) \to \mathbb{R}_+$ is a $\mathcal{K}_{[0,d)}$-function if it is continuous, strictly increasing, and vanishes at zero. A function $\alpha: \mathbb{R}_+ \to \mathbb{R}_+$ is a $\mathcal{K}_\infty$-function if it is a $\mathcal{K}$-function and also satisfies $\alpha(r) \to \infty$ as $r \to \infty$. A function $\beta: \mathbb{R}_+ \times \mathbb{R}_+ \to \mathbb{R}_+$ is a $\mathcal{KL}$-function if for any fixed $t \geq 0$, $\beta(\cdot,t)$ is a $\mathcal{K}$-function, and for any fixed $r \geq 0$, $\beta(r,\cdot)$ is decreasing and $\beta(r,t) \to 0$ as $t \to \infty$. 
\end{definition}

\begin{lemma}[The cyclic property of the trace, equation (16) in \cite{book_Petersen} ]\label{lm:traceCyclic}
    For any $X,Y,Z \in \mathbb{R}^{n \times n}$, $\Tr{XYZ} = \Tr{ZXY} = \Tr{YZX}$.
\end{lemma}

\begin{lemma}[Trace inequality \cite{Wang1986}]\label{lm:traceIneq}
    Let $S\in \mathbb{R}^{n \times n}$ be real symmetric and $P \in  \mathbb{R}^{n \times n}$ be real symmetric and positive semidefinite. Then,
    \begin{align*}
        \eigmin{S} \Tr{P} \le \Tr{SP} \le \eigmax{S} \Tr{P}.
    \end{align*}
\end{lemma}

\begin{lemma}[Cauchy-Schwarz inequality]\label{lm:CSInequality}
    For any $K_1,K_2 \in \mathbb{R}^{m\times n}$, $R \in \mathbb{P}^{m}$, and $Y \in \mathbb{P}^n$, we have
    \begin{align}\label{eq:CSInequality}
        \innprod{K_1}{RK_2}_Y \leq \sqrt{ \innprod{K_1}{RK_1}_Y}\sqrt{ \innprod{K_2}{RK_2}_Y}. 
    \end{align}
\end{lemma}
\begin{proof}
    \LC{Inequality \eqref{eq:CSInequality} is obviously true if $K_2 = 0$. If $K_2 \neq 0$, define} $K_3$ as
    \begin{align}\label{eq:orthoDecom}
        K_3 = K_1 - \frac{\innprod{K_1}{RK_2}_Y}{\innprod{K_2}{RK_2}_Y}K_2.
    \end{align}
    It is clear that $\innprod{K_3}{RK_2}_Y=0$. Therefore, by plugging \eqref{eq:orthoDecom} into $\innprod{K_1}{RK_1}_Y$, we can obtain
    \begin{align}\label{eq:CSInequality1}
        \innprod{K_1}{RK_1}_Y = \innprod{K_3}{RK_3}_Y + \frac{\innprod{K_1}{RK_2}_Y^2}{\innprod{K_2}{RK_2}_Y^2} \innprod{K_2}{RK_2}_Y\ge \frac{\innprod{K_1}{RK_2}_Y^2}{\innprod{K_2}{RK_2}_Y}.
    \end{align}
    Hence, \eqref{eq:CSInequality} readily follows from \eqref{eq:CSInequality1}.
\end{proof}

\begin{lemma}\label{lm:homomorphism}
   The map $h(v) = \frac{1}{1+\norm{v}}v: \mathbb{R}^m \to \mathcal{W}:=\{w\in \mathbb{R}^m| \norm{w} < 1\}$ is a homeomorphism.
\end{lemma}
\begin{proof}
    For any $w \in \mathcal{W}$, let $g(w) = \frac{1}{1-\norm{w}}w$. Clearly, for any $v \in \mathbb{R}^m$, $g(h(v)) = v$ and for any $w \in \mathcal{W}$, $h(g(w)) = w$. Hence, $g$ is the inverse function of $h$, i.e. $g = h^{-1}$. Since both $h$ and $g$ are continuous, $h$ is a homeomorphism.    
\end{proof}

\begin{lemma}\label{lm:radiallyUnbounded}
    Suppose $\omega_1, \, \omega_2: \mathbb{R}^n \to \mathbb{R}$ are continuous, positive definite with respect to $\chi^*$, and radially unbounded. Then, there exist $\mathcal{K}_\infty$-functions $\rho_1$ and $\rho_2$ such that
    \begin{align*}
        \rho_1(\omega_2(\chi)) \le \omega_1(\chi) \le \rho_2(\omega_2(\chi)), \quad \forall \chi \in \mathbb{R}^n.
    \end{align*}
\end{lemma}
\begin{proof}
    The proof follows from \cite[Proposition 2.6]{Sontag2022} by considering the open subset as $\mathbb{R}^n$ and the compact set as $\{\chi^*\}$.
\end{proof}

\begin{lemma}[Weak triangle inequality in \cite{Jiang1994}]\label{lm:weakTriangle}
    For any $\mathcal{K}$-function $\alpha$, any $\mathcal{K}_\infty$-function $\rho$, and any nonnegative real numbers $a$ and $b$, we have
    \begin{align*}
        \alpha(a+b) \le \alpha \circ (\mathrm{Id} + \rho)(a) + \alpha \circ (\mathrm{Id} + \rho^{-1}) (b).
    \end{align*}
\end{lemma}

\begin{lemma}[Theorem 18 in \cite{book_sontag}]\label{lm:LyaEquaIntegral}
    If $A \in \mathbb{R}^{n \times n}$ is Hurwitz, then the Lyapunov equation 
    \begin{align*}
        A^T P + PA + Q = 0
    \end{align*}
    has a unique solution for any $Q \in \mathbb{R}^{n \times n}$, and the solution can be expressed as
    \begin{align*}
        P = \int_{0}^\infty e^{A^T t} Q e^{At} \mathrm{d}t.
    \end{align*}
\end{lemma}

\section{Small-Disturbance Input-to-State Stability}
Let $\mathcal{S}$ denote an open subset of $\mathbb{R}^n$, which will be called the admissible set of states. Consider the following nonlinear system
\begin{align}\label{eq:nonLinearSys}
    \dot{\chi}(t) = f(\chi(t), w(t)),
\end{align}
where $f: \mathcal{S} \times \mathbb{R}^m \to \mathbb{R}^n$ is \LC{a smooth function}, and inputs $w: \mathbb{R}_+ \to \mathbb{R}^m$ are measurable and locally essentially bounded functions. Assume the unforced system has an equilibrium $\chi^*$, i.e. $f(\chi^*,0)=0$.

\begin{definition}[Definition 2.1 in \cite{Sontag2022}]\label{def:sizeFunc}
    A function $\mathcal{V}: \mathcal{S} \to \mathbb{R}_+$ is a size function for $(\mathcal{S},\chi^*)$ if $\mathcal{V}$ is 
    \begin{enumerate}
        \item continuous;
        \item positive definite with respect to $\chi^*$, i.e. $\mathcal{V}(\chi^*)=0$ and $\mathcal{V}(\chi)>0$ for all $\chi \neq \chi^*$,  $\chi \in \mathcal{S}$;
        \item coercive, i.e. for any sequence $\{\chi_k\}_{k=0}^\infty$, $\chi_k \to \partial \mathcal{S}$ or $\norm{\chi_k} \to \infty$, it holds that $\mathcal{V}(\chi_k) \to \infty$, as $k \to \infty$.
    \end{enumerate}
\end{definition}

\begin{definition}\cite{Pang_Jiang_2021,Pang2022}\label{def:smallISS}
System \eqref{eq:nonLinearSys} is small-disturbance input-to-state stable (ISS) if there exist a size function $\mathcal{V}$, a constant $d > 0$ (possibly $\infty$), a $\mathcal{KL}$-function $\beta$, and a $\mathcal{K}_{[0,d)}$-function $\gamma$, such that for all inputs $w$ essentially bounded by $d$ (i.e. $\norm{w}_\infty<d$), and all initial states $\chi(0) \in \mathcal{S}$, $\chi(t) $ remains in $\mathcal{S}$ and satisfies
\begin{align}\label{eq:ISS}
     \mathcal{V}(\chi(t)) \leq \beta(\mathcal{V}(\chi(0)),t) + \gamma(\norm{w}_\infty), \quad \forall t \ge 0.
\end{align}
\end{definition}

As shown in \cite{sontag1995characterizations}, by causality, the same definition would result if one would replace $\norm{w}_\infty$ by $\norm{w_t}_\infty$  in \eqref{eq:ISS}.

\begin{definition}\label{def:smallISSLyap}
A continuously differentiable function $\mathcal{V}:\mathcal{S} \to \mathbb{R}$ is a small-disturbance ISS-Lyapunov function for system \eqref{eq:nonLinearSys} if 
\begin{enumerate}
    \item $\mathcal{V}$ is a size function for $(\mathcal{S},\chi^*)$;  
    \item there exist a $\mathcal{K}$-function $\alpha_1$ and a continuous and positive definite function $\alpha_2$ such that if $\norm{\mu} \leq \alpha_1(\mathcal{V}(\chi))$,
    \begin{align}\label{eq:dissipativeLya}
        \nabla\mathcal{V}(\chi)^Tf(\chi,\mu) \leq - \alpha_2(\mathcal{V}(\chi)).
    \end{align}
\end{enumerate}
\end{definition}

\begin{theorem}\label{thm:smallISSSufficient}
System \eqref{eq:nonLinearSys} is small-disturbance ISS if and only if it admits a small-disturbance ISS-Lyapunov function.
\end{theorem}
\begin{proof} \textbf{Sufficiency:} 
This is an adaptation of the proof of the analogous result for the ISS property~\cite{Sontag1989}.
Let $d= \sup_{r \in \mathbb{R}_+} \alpha_1(r)$.  For any inputs $w$ with $\norm{w}_\infty<d$, define the sublevel set $\mathcal{S}_c = \{\chi \in \mathcal{S}| \mathcal{V}(\chi) \leq c \}$, where $c = \alpha_1^{-1}(\norm{w}_\infty)$. Borrowing techniques similar to those in the proof of \cite[Theorem 1]{Sontag1989}, we can show  that $\mathcal{S}_c$ is forward invariant, i.e. if $\chi(t_0) \in \mathcal{S}_c$ for some $t_0 \ge 0$, then $\chi(t) \in \mathcal{S}_c$ for all $t \ge t_0$.

Now, let $t_1 = \inf\{t \in \mathbb{R}_+| \chi(t) \in \mathcal{S}_c \} \leq \infty$. Therefore, for any $t \ge t_1$, we have
\begin{align} \label{eq:sizefunctionUpperBound1}
    \mathcal{V}(\chi(t)) \le \alpha_1^{-1}(\norm{w}_\infty).
\end{align}
For $t < t_1$, $\alpha_1(\mathcal{V}(\chi(t))) \ge \norm{w}_\infty$, which implies that
\begin{align*}
    \frac{\de \mathcal{V}(\chi(t))}{\de t} \le -\alpha_2(\mathcal{V}(\chi(t))) , \quad \forall t < t_1.
\end{align*}
Hence, $\mathcal{V}(\chi(t)) \le \mathcal{V}(\chi(0))$, $\forall t < t_1$. By the comparison principle \cite[Lemma 4.4]{Sontag_SIAM_1996}, there exists a $\mathcal{KL}$-function $\beta$ such that 
\begin{align}\label{eq:sizefunctionUpperBound2}
    \mathcal{V}(\chi(t)) \le \beta(\mathcal{V}(\chi(0)),t), \quad \forall t < t_1.
\end{align}
Combining \eqref{eq:sizefunctionUpperBound1} and \eqref{eq:sizefunctionUpperBound2}, the small-disturbance ISS property \eqref{eq:ISS} follows readily with $\gamma = \alpha_1^{-1}$.

\textbf{Necessity:} We first prove the case when $\mathcal{S} = \mathbb{R}^n$. Reparameterize the input as
\begin{align*}
    w(t) = \frac{d}{d+\norm{v(t)}}v(t) =: h(v(t)).
\end{align*}
It is shown in Lemma \ref{lm:homomorphism} that $h$ is a homeomorphism from $\mathbb{R}^m$ to $\mathcal{W} := \{w \in \mathbb{R}^m|\norm{w} <d\}$ with $v(t) = h^{-1}(w(t))=\frac{d}{d - \norm{w(t)}}w(t)$. With the input change, we have
\begin{align}\label{eq:nonlinearSysRepara}
    \dot{\chi}(t) = f\left(\chi(t), h(v(t))\right) =: f_1\left(\chi(t), v(t)\right).
\end{align}
Since $\gamma$ is a $\mathcal{K}_{[0,d)}$-function and $\gamma_1(r) = \frac{dr}{d+r}$ is a $\mathcal{K}$-function with the range $[0,d)$, $\gamma_2 = \gamma \circ \gamma_1$ is a $\mathcal{K}$-function.
According to \eqref{eq:ISS}, it holds
\begin{align}\label{eq:ISSreparameter}
     \mathcal{V}(\chi(t)) \leq \beta(\mathcal{V}(\chi(0)),t) + \gamma_2(\norm{v}_\infty).
\end{align}    
Since $ \mathcal{V}$ is a size function for $(\mathbb{R}^n,\chi^*)$, according to Lemma \ref{lm:radiallyUnbounded}, there exist $\mathcal{K}_\infty$-functions $\rho_1$ and $\rho_2$, such that
\begin{align}\label{eq:sizeBound}
     \rho_1\left(\norm{\chi-\chi^*}\right) \le  \mathcal{V}(\chi) \le \rho_2\left(\norm{\chi-\chi^*}\right), \quad \forall \chi \in \mathbb{R}^n.
\end{align}
Plugging \eqref{eq:sizeBound} into \eqref{eq:ISSreparameter} yields
\begin{align*}
    \norm{\chi(t)-\chi^*} \le \rho_1^{-1}\left[\beta\left(\rho_2(\norm{\chi(0)-\chi^*}),t\right) + \gamma_2(\norm{v}_\infty) \right].
\end{align*}
By Lemma \ref{lm:weakTriangle}, there exist a $\mathcal{KL}$-function $\beta_1$ and a $\mathcal{K}$-function $\gamma_3$, such that 
\begin{align*}
    \norm{\chi(t)-\chi^*} \le \beta_1\left(\norm{\chi(0)-\chi^*}),t\right) + \gamma_3(\norm{v}_\infty).
\end{align*}
Hence, system \eqref{eq:nonlinearSysRepara} is ISS with respect to $v$. According to \cite[Theorem 1]{sontag1995characterizations}, there exists an ISS-Lyapunov function $\mathcal{V}_1$, such that
\begin{align*}
    \rho_3(\norm{\chi-\chi^*}) \le \mathcal{V}_1(\chi) \le \rho_4(\norm{\chi-\chi^*}), \quad \forall \chi \in \mathbb{R}^n,
\end{align*}
and
\begin{align}\label{eq:V1ISS}
    \nabla \mathcal{V}_1(\chi)^Tf_1(\chi,v) \le - \rho_6(\norm{\chi-\chi^*}),
\end{align}
for any $\chi \in \mathbb{R}^n$ and any $v \in \mathbb{R}^m$ satisfying $\norm{v} \le \rho_5(\norm{\chi-\chi^*})$, where $\rho_{i}$ ($i$=3,4,5) are $\mathcal{K}_\infty$-functions, and $\rho_{6}$ is a $\mathcal{K}$-function. This, in turn, implies that if $\norm{w} \le \gamma_1 \circ \rho_5 \circ \rho_4^{-1} (\mathcal{V}_1(\chi)) $, equation \eqref{eq:V1ISS} holds. 
Since $\gamma_1\circ \rho_5 \circ \rho_4^{-1}$ is a $\mathcal{K}$-function with the range $[0,d)$, we obtain that $\mathcal{V}_1$ is a small-disturbance ISS-Lyapunov function.

Next, we prove the general case when $\mathcal{S}$ is an open subset of $\mathbb{R}^n$. By assumption, for the unforced system $\dot{\chi}(t) = f(\chi(t),0)$, $\chi^*$ is an asymptotically stable point and the domain of stability is $\mathcal{S}$. By \cite[Theorem 2.2]{WILSON1967323} and the Brown–Stallings Theorem \cite{Milnor1964}, $\mathcal{S}$ is diffeomorphic to $\mathbb{R}^n$. We denote by $\varphi: \mathcal{S} \to \Rn$ the diffeomorphism, $\zeta = \varphi(\chi)$, and $\zeta^* = \varphi(\chi^*)$. As a consequence, we have
\begin{align}\label{eq:zetaSystem}
    \dot{\zeta}(t) = J_\varphi(\varphi^{-1}(\zeta(t)))f(\varphi^{-1}(\zeta(t)),w(t)) =: f_2(\zeta(t), w(t)),
\end{align}
where $J_\varphi(\chi)$ is the Jacobian matrix of $\varphi(\chi)$. Since system \eqref{eq:nonLinearSys} is small-disturbance ISS over $\mathcal{S} \times \mathbb{R}^m$, system \eqref{eq:zetaSystem} is small-disturbance ISS over $\Rn \times \mathbb{R}^m$. Then, according to the conclusion from the case of $\mathcal{S} = \Rn$, there exists a small-disturbance ISS-Lyapunov function $\mathcal{V}_1: \Rn \to \R_+$ for system \eqref{eq:zetaSystem} with respect to $\zeta^*$. It readily follows that $\mathcal{V}_1\circ\varphi: \mathcal{S} \to \R_+$ is a small-disturbance ISS-Lyapunov function for system \eqref{eq:nonLinearSys}.
Therefore, the necessity holds.
\end{proof}

\begin{remark}
    It should be mentioned that small-disturbance ISS is not equivalent to the notion of integral ISS \cite{MR1629012}. Consider a scalar nonlinear system
    $\dot{\chi}(t) = - \frac{\chi(t)}{1+\chi(t)^2} + w(t)$.
    Differentiating the Lyapunov function $\mathcal{V}(\chi) = \log(1+\chi^2)$ with respect to time yields 
    $\dot {\mathcal{V}} \leq \frac{-2\chi^2}{(1+\chi^2)^2} + |w|$, which implies that the system is integral ISS \cite{angeli_sontag_wang_iISS_TAC00}. However, for any arbitrarily small input $0<\bar{w} <0.5$, the trajectories of the system diverge whenever the initial conditions $\chi(0) > \frac{1+\sqrt{1-4\bar{w}^2}}{2\bar{w}}$. Hence, the system is not small-disturbance ISS.
\end{remark}

\section{Robustness Analysis of Perturbed Gradient Flows}
Consider the following constrained nonlinear programming problem
\begin{align}\label{eq:optiPro}
    &\min_{z \in \mathcal{Z}} \mathcal{J}(z) 
\end{align}    
where $\mathcal{Z}$ is an open subset of $\mathbb{R}^n$, which is called an admissible set of variables; $\mathcal{J}: \mathcal{Z} \to \mathbb{R}$ is an objective function with a global minimizer $z^*$. 

\begin{definition}\label{def:ProperLossFun}
    A function $\mathcal{J}: \mathcal{Z} \to \mathbb{R}$ is a proper objective function if
    \begin{enumerate}
        \item \LC{$\mathcal{J}(z)$ is a smooth function};
        \item $\mathcal{J}(z) - \mathcal{J}(z^*)$ is a size function for $(\mathcal{Z}, z^*)$;
        \item there exists a $\mathcal{K}$-function $\alpha_3$, such that $\norm{\nabla \mathcal{J}(z)} \ge \alpha_3(\mathcal{J}(z) - \mathcal{J}(z^*))$ (CJS-PL estimate).
    \end{enumerate}
\end{definition}

The perturbed gradient flow for \eqref{eq:optiPro} is 
\begin{align}\label{eq:gradientflowopti}
    \dot{z}(t) = -\eta \nabla \mathcal{J}(z(t)) + e(t),
\end{align}
where $\eta > 0$ is a constant, and $e \in \mathcal{L}_\infty^{n}$ denotes the perturbation to the gradient flow. The perturbation $ e(t)$ may arise from inaccurate gradient estimation for data-driven optimization or arithmetic rounding errors of numerical computation. 

\begin{theorem}\label{eq:gradFlowISS}
    If $\mathcal{J}$ is a proper objective function, then system \eqref{eq:gradientflowopti} is small-disturbance ISS.
\end{theorem}
\begin{proof}
    We will prove that $\mathcal{V}_2(z) = \mathcal{J}(z) - \mathcal{J}(z^*)$ is a small-disturbance ISS-Lyapunov function. Firstly, notice that $\mathcal{V}_2$ is a size function for $(\mathcal{Z}, z^*)$. Then, by the Cauchy-Schwarz inequality and Young's inequality, it holds 
    \begin{align}\label{eq:derivativeVo1}
    \begin{split}
       &\nabla{\mathcal{V}}_2(z)^T(-\eta \nabla \mathcal{J}(z) + e)  = - \eta \norm{ \nabla \mathcal{J}(z)}^2 + \nabla \mathcal{J}(z)^T e \\
        &\leq - \frac{\eta}{2} \norm{ \nabla \mathcal{J}(z)}^2 +  \frac{1}{2\eta}\norm{e}^2 \leq - \frac{\eta}{2} \alpha_3^2(\mathcal{V}_2(z)) + \frac{1}{2\eta}\norm{e}^2,
    \end{split}
    \end{align}
    where the last inequality is a direct consequence of the CJS-PL property.
    Hence, if $\norm{e} \leq \frac{\eta}{\sqrt{2}}\alpha_3(\mathcal{V}_2(z))$, it follows from \eqref{eq:derivativeVo1} that
    \begin{align*}
       \nabla{\mathcal{V}}_2(z)^T(-\eta \nabla \mathcal{J}(z) + e) \leq -\frac{\eta}{4}\alpha_3^2(\mathcal{V}_2(z)).
    \end{align*}
    Since $\alpha_3$ is a $\mathcal{K}$-function, $\frac{\eta}{\sqrt{2}}\alpha_3$ is also a $\mathcal{K}$-function. Therefore, by Definition \ref{def:smallISSLyap}, $\mathcal{V}_2(z) = \mathcal{J}(z) - \mathcal{J}(z^*)$ is a small-disturbance ISS-Lyapunov function. According to Theorem \ref{thm:smallISSSufficient}, we conclude that system \eqref{eq:gradientflowopti} is small-disturbance ISS.
\end{proof}

\begin{remark}
Suppose that one has the classical PL inequality, namely
for each $r>0$, there is a $c_r > 0$ such that $\norm{\nabla \mathcal{J}(z)}^2 \ge c_r (\mathcal{J}(z)-\mathcal{J}(z^*))$ for all $z$ in the sublevel set $\{z\in \mathcal{Z}|\mathcal{J}(z)-\mathcal{J}(z^*) \le r\} $. Then there is some positive definite function $\alpha_3$ such that $\norm{\nabla \mathcal{J}(z)} \ge \alpha_3(\mathcal{J}(z)-\mathcal{J}(z^*))$ for all $z \in \mathcal{Z}$. By following \cite{angeli_sontag_wang_iISS_TAC00} and \eqref{eq:derivativeVo1}, it shows that the classical PL condition gives integral ISS.
\end{remark}

\begin{remark}
As discussed in the Introduction, the gradient dominance condition given by the CJS-PL estimate is weaker than its counterpart in \cite[Definition 4.1]{Sontag2022}, where $\alpha_3$ is required to be a $\mathcal{K}_\infty$-function. In addition, the CJS-PL estimate implies that the perturbed gradient flow \eqref{eq:gradientflowopti} is integral ISS.
\end{remark}

\section{Application to LQR Problem}
\subsection{Preliminaries of LQR}
Consider the following linear time-invariant system
\begin{align*}
    \dot{x}(t) = Ax(t) + Bu(t), \,\, x(0) = x_0,
\end{align*}
where $x(t) \in \mathbb{R}^n$ is the state; $x_0$ is the initial state; $u(t) \in \mathbb{R}^{m}$ is the control input; $A$ and $B$ are constant matrices with compatible dimensions. The continuous-time LQR aims at finding a state-feedback controller by solving the following optimal control problem
\begin{align}\label{eq:LQRcost}
    \min_{u \in \mathcal{L}_\infty^{m}}\mathcal{J}_{1}(x_0, u) =  \int_{0}^{\infty} x^T(t)Qx(t) + u^T(t) R u(t) \de t , 
\end{align}
with $Q = Q^T \succeq 0$, and $R=R^T \succ 0$. Under the assumption that $(A,B)$ is stabilizable and $(A,\sqrt{Q})$ is observable, as shown in \cite[Section 8.4]{book_sontag}, the optimal controller is
\begin{align}\label{eq:Koptexpression}
    u^*(x(t)) = -{K^*}x(t), \quad K^* = R^{-1}B^TP^*,
\end{align}
where $P^* = (P^*)^T$ is the unique positive definite solution of the following algebraic Riccati equation (ARE)
\begin{align}\label{eq:continuousARE}
    A^T P^* + P^* A + Q - P^*BR^{-1}B^TP^* = 0.
\end{align}

Let $\mathcal{G} = \{K \in \mathbb{R}^{m \times n}| A-BK \text{ is Hurwitz} \}$ denote the admissible set of all stabilizing control gains. For any stabilizing controller $u(t) = -Kx(t)$, where $K \in \mathcal{G}$, and any nonzero initial state $x_0 \in \mathbb{R}^n$, the corresponding cost is
\begin{align*}
    \mathcal{J}_{1}(x_0, K) = \int_{0}^\infty x_0^T e^{(A-BK)^Tt} (Q + K^TRK) e^{(A-BK) t} x_0 \de t = x_0^T P_K x_0,
\end{align*}
where $P_K = P_K^T$ is the unique positive definite solution of the following Lyapunov equation
\begin{align}\label{eq:continuousLyapunov}
    (A-BK)^{T}P_K + P_K (A-BK) + Q + K^T R K = 0.
\end{align}
Since $K^*$ is the optimal control gain and $P^* = P_{K^*}$ is the minimal cost matrix, by \cite[page 382]{book_sontag}, it holds
\begin{align}\label{eq:PoptMinimum}
\mathcal{J}_1(x_0,K) = x^T_0 P_K x_0 \ge x^T_0 P^* x_0 = \mathcal{J}_1(x_0,K^*), \quad  \forall x_0 \in \mathbb{R}^n.
\end{align}
This implies that $P_K \succeq P^*$ for all $K \in \mathcal{G}$.  

Since the objective function $\mathcal{J}_1$ of the LQR problem depends on the initial condition, we are motivated to study an equivalent optimization problem ($\min_{K \in \mathcal{G}}\mathcal{J}_2(K)$), which is independent of the initial condition. For any initial state, an upper bound for $\mathcal{J}_1(x_0,K)$ is 
\begin{align*}
    \mathcal{J}_1(x_0,K) \le \norm{x_0}^2\Tr{P_K} = \norm{x_0}^2\mathcal{J}_2(K),
\end{align*}
where 
\begin{align}\label{eq:costJc_closedform}
    \mathcal{J}_2(K) := \Tr{P_K},
\end{align}
which is independent of $x_0$ \LC{and is an analytic function \cite[Proposition 3.2]{bu2020policy}.} Since $P_K \succeq P^*$ for any $K \in \mathcal{G}$, $\mathcal{J}_2(K) \ge \mathcal{J}_2(K^*)$. In addition, since $\Tr{P_K} = \Tr{P^*}$ implies $P_K = P^*$ and $K=K^*$, $\mathcal{J}_2(K)$ has a unique minimum at $K^*$. Thus, the optimal control gain $K^*$ can be obtained by the following policy optimization problem 
\begin{align*}
    \min_{K \in \mathcal{G}}\mathcal{J}_2(K).    
\end{align*}

Before calculating $\nabla \mathcal{J}_2(K)$, which denotes the gradient of $\mathcal{J}_2(K)$ over the Euclidean space, let us define the matrix $Y_K \in \mathbb{P}^n$ as the solution of
\begin{align}\label{eq:YKDef}
    (A-BK)Y_K + Y_K (A-BK)^T + I_n = 0.
\end{align}
It is noticed that according to \cite[Lemma 3.18]{book_Zhou}, $Y_K \succ 0$ for any $K \in \mathcal{G}$. In addition, $Y^*$ is defined as the solution of \eqref{eq:YKDef} with $K$ replaced by $K^*$.
Since $A-BK$ is Hurwitz, by Lemma \ref{lm:LyaEquaIntegral}, $Y_K$ can be expressed as
\begin{align}\label{eq:YKexpress}
    Y_K = \int_{0}^{\infty} e^{(A-BK)t}e^{(A-BK)^Tt} \de t.
\end{align}

\begin{lemma}\label{lm:J2Expansion}
    For any $K \in \mathcal{G}$, when it is perturbed by $E$ with $K+E\in \mathcal{G}$ (recall that $\mathcal{G}$ is an open set), the second-order Taylor series approximation of $\mathcal{J}_2(K + E)$ is
    \begin{align}\label{eq:J2Expansion}
    \begin{split}
        \mathcal{J}_2(K+E ) &= \mathcal{J}_2(K )+ 2\Tr{E^T (RK - B^TP_K)Y_K } + \Tr{E^T R EY_K} \\
        &+ 2\Tr{E^T (RK - B^TP_K) \Delta Y_K }+O(\norm{E}_F^3) ,
    \end{split}
    \end{align} 
    
    where
    \begin{align}\label{eq:DeltaYK}
        \Delta Y_K := -\int_{0}^{\infty} e^{(A-BK )t}  ( BE Y_K + Y_K E^T B^T) e^{(A-BK )^{T}t},
    \end{align}
    and $O(\norm{E}_F^3)$ is the remainder of the approximation. 
\end{lemma}
\begin{proof}
    Firstly, we calculate $Y_{K+E}$ when $K$ is perturbed to $K+ E$. Using \eqref{eq:YKDef}, we have
    \begin{align}\label{eq:YKDeltaK}
        (A-BK - BE)Y_{K+E} + Y_{K+E}(A-BK -BE)^T + I_n = 0.
    \end{align}
    Subtracting \eqref{eq:YKDef} from \eqref{eq:YKDeltaK}, we have
    \begin{align*}
        &(A-BK -BE) (Y_{K+E} - Y_K) +  (Y_{K+E} - Y_K) (A-BK -BE )^T \\
        &- BE Y_K - Y_K E^T B^T = 0.
    \end{align*}
    Let $\Delta Y_K$ denote the first-order term in the Taylor expansion of $Y_{K+E}$, i.e. $Y_{K+E} = Y_K + \Delta Y_K + O(\norm{E}_F^2)$. Then, $\Delta Y_K$ satisfies 
    \begin{align*}
        (A-BK) \Delta Y_K +  \Delta Y_K (A-BK  )^T - BE Y_K - Y_K E^T B^T = 0,
    \end{align*}    
    which, in turn, implies \eqref{eq:DeltaYK}.

    Then, we will calculate $P_{K + E}$ for the perturbed control gain $K+E$. By \eqref{eq:continuousLyapunov}, we have
    \begin{align}\label{eq:PKDeltaK}
    \begin{split}
    &(A-BK-BE)^{T}P_{K + E} + P_{K + E} (A-BK-BE) \\
    &+ Q + (K+E)^T R (K+E) = 0.        
    \end{split}
    \end{align}
    Subtracting \eqref{eq:continuousLyapunov} from \eqref{eq:PKDeltaK} yields
    \begin{align*}
    \begin{split}
    &(A-BK-BE)^{T}(P_{K + E}-P_K) + (P_{K + E}-P_K)  (A-BK-BE)\\
    &+  E^T (RK - B^TP_K) + (RK - B^TP_K)^T E + E^T R E= 0,
    \end{split}
    \end{align*}
    which is equivalent to
    \begin{align}\label{eq:PKE-PK}
    \begin{split}
        &P_{K + E}-P_K = \int_{0}^{\infty} e^{(A-BK-BE)^Tt} [ E^T (RK - B^TP_K) \\
        &+ (RK - B^TP_K)^T E + E^T R E] e^{(A-BK-BE)t}\de t.
    \end{split}
    \end{align}
    
    Taking the trace of \eqref{eq:PKE-PK}, considering $\mathcal{J}_2(K+E) -\mathcal{J}_2(K) = \Tr{P_{K + E}-P_K}$, and using the cyclic property of the trace in Lemma \ref{lm:traceCyclic}, we can obtain
    \begin{align}\label{eq:JKdiff}
    \begin{split}
        &\mathcal{J}_2(K+E ) -\mathcal{J}_2(K) 
        \\
        &=2\Tr{E^T (RK - B^TP_K)Y_{K+E}} + \Tr{E^T R EY_{K+E}}.
    \end{split}
    \end{align}
    Since the first-order Taylor series approximation of $Y_{K+E}$ is $Y_K + \Delta Y_K$, plugging it into \eqref{eq:JKdiff}, we can obtain \eqref{eq:J2Expansion}.
\end{proof}

In \eqref{eq:J2Expansion}, the first-order term is $2\Tr{E^T (RK - B^TP_K)Y_K } = \innprod{E}{2(RK - B^TP_K)Y_K}$, which should be equal to $\innprod{E}{\nabla \mathcal{J}_2(K)}$ by Taylor expansion. In other words, $\innprod{E}{2(RK - B^TP_K)Y_K} = \innprod{E}{\nabla \mathcal{J}_2(K)}$ for any small perturbation $E \in \mathbb{R}^{m \times n}$. Hence, the gradient of the objective function $\mathcal{J}_2(K)$ is 
\begin{align}\label{eq:JcGradient}
    \nabla \mathcal{J}_2(K) = 2(RK - B^TP_K)Y_K.
\end{align}

We claim that $\mathcal{J}_2(K)$ is a proper objective function (Definition \ref{def:ProperLossFun}). To this end, we will first prove several intermediate lemmas providing bounds on $Y_K$ and $P_K$. The following lemma gives a lower bound of $\eigmin{Y_K}$.

\begin{lemma}\label{lm:eigminYKLowerBound}
For any $K \in \mathcal{G}$, we have
\begin{align*}
    \eigmin{Y_K} \geq \frac{1}{2\norm{A-BK^*}_F + 2\norm{B}\norm{K-K^*}_F},
\end{align*}     
where $K^*$ is the optimal control gain in \eqref{eq:Koptexpression}.
\end{lemma}
\begin{proof}
    Let $q \in \mathbb{R}^n$ denote a unit eigenvector of $Y_K$ associated with the eigenvalue $\eigmin{Y_K}$. Pre and post multiplying \eqref{eq:YKDef} by $q^T$ and $q$, respectively, we have
    \begin{align}\label{eq:A-BKeig1}
        \eigmin{Y_K}q^T[(A-BK) + (A-BK)^T]q = -1.
    \end{align}
    Since $Y_K \succ 0$, it follows from \eqref{eq:A-BKeig1} that $q^T[(A-BK) + (A-BK)^T]q<0$. Consequently, we have
    \begin{align}\label{eq:A-BKeig2}
    \begin{split}
        &\eigmin{(A-BK) + (A-BK)^T} = \min_{\norm{\nu}=1} \nu^T[(A-BK) + (A-BK)^T]\nu \\
        &\leq q^T[(A-BK) + (A-BK)^T]q<0.
    \end{split}
    \end{align}
    where the first equality is obtained by the Rayleigh principle \cite[Theorem 4.2.2]{book_Horn}. By \eqref{eq:A-BKeig1} and \eqref{eq:A-BKeig2}, it holds
    \begin{align*}
    \begin{split}
        &\eigmin{Y_K} = \frac{1}{-q^T[(A-BK) + (A-BK)^T]q} \\
        &\geq \frac{1}{- \eigmin{(A-BK) + (A-BK)^T}} \\
        &\geq \frac{1}{2\norm{A-BK^*}_F + 2 \norm{B} \norm{K-K^*}_F},
    \end{split}
    \end{align*}
   where the last line uses the relation $\norm{B(K^* - K)} \le \norm{B}\norm{(K^* - K)}\le \norm{B}\norm{(K^* - K)}_F$. Thus, the proof is completed.
\end{proof}

The following lemma gives the bounds of $\Tr{Y_K}$.
\begin{lemma}\label{lm:YKBound}
    Given $K \in \mathcal{G}$ and $Q \succ 0$, we have
    \begin{align}\label{eq:YKBound}
        \frac{\Tr{P_K-P^*}}{\norm{R}\norm{K-K^*}_F^2} \le \Tr{Y_{K}} \le  \frac{\Tr{P_K}}{\eigmin{Q}}
    \end{align}
\end{lemma}
\begin{proof}    
    Considering $K^* = R^{-1}B^TP^*$, we can rewrite \eqref{eq:continuousARE} as
    \begin{align}\label{eq:AoptPopt}
        (A-BK^*)^T P^* + P^* (A-BK^*) + Q + (K^*)^T R K^* = 0.
    \end{align}
   It follows from \eqref{eq:Koptexpression} and \eqref{eq:AoptPopt} that
    \begin{align}\label{eq:AoptPoptREwrite}
    \begin{split}
        &(A-BK)^TP^* + P^*(A-BK) + Q + (K^*)^TRK^* \\
        &+ (K-K^*)^TRK^* + (K^*)^TR(K-K^*) = 0.
    \end{split}
    \end{align}
    Subtracting \eqref{eq:AoptPoptREwrite} from \eqref{eq:continuousLyapunov} and completing the squares yield
    \begin{align*}
        (A-BK)^T(P_K - P^*) +(P_K - P^*)(A-BK) + (K-K^*)^TR(K-K^*) = 0.
    \end{align*}  
    Since $A-BK$ is Hurwitz, by Lemma \ref{lm:LyaEquaIntegral}, we have
    \begin{align}\label{eq:PKdiff1}
        P_K - P^* = \int_{0}^{\infty} e^{(A-BK)^{T}t} (K-K^*)^TR(K-K^*) e^{(A-BK)t} \de t.
    \end{align}
    Taking the trace of \eqref{eq:PKdiff1} and using the cyclic property of the trace in Lemma \ref{lm:traceCyclic} and \eqref{eq:YKexpress}, we obtain 
    \begin{align}\label{eq:TrPKPopt}
    \begin{split}
        \Tr{P_K - P^*} &= \Tr{\int_{0}^{\infty}e^{(A-BK)t} e^{(A-BK)^{T}t}  \de t (K-K^*)^TR(K-K^*) } \\
        &= \Tr{Y_K(K-K^*)^TR(K-K^*)}.
    \end{split}
    \end{align}   
    By the trace inequality in Lemma \ref{lm:traceIneq} and considering the following relation
    \begin{align*}
        \norm{(K-K^*)^TR(K-K^*)} \le \Tr{(K-K^*)^TR(K-K^*)} \le \norm{R} \norm{K-K^*}_F^2,
    \end{align*}
    we have
    \begin{align*}
        \Tr{P_K - P^*} \le \norm{R}\norm{K-K^*}_F^2 \Tr{Y_K}.
    \end{align*}
    Hence, the lower bound of $\Tr{Y_K}$ in \eqref{eq:YKBound} is obtained.
    
    Since $A-BK$ is Hurwitz, it follows from Lemma \ref{lm:LyaEquaIntegral} and \eqref{eq:continuousLyapunov} that 
    \begin{align}\label{eq:PKIntofABK}
        P_K = \int_{0}^{\infty} e^{(A-BK)^Tt}(Q + K^T R K) e^{(A-BK)t} \de t.
    \end{align}
    Taking the trace of \eqref{eq:PKIntofABK}, and again using the cyclic property of the trace in Lemma \ref{lm:traceCyclic} and the trace inequality in Lemma \ref{lm:traceIneq}, we have

\begin{align*}
\begin{split}
    \Tr{P_K} &\ge \Tr{\int_{0}^{\infty} e^{(A-BK)^Tt}Qe^{(A-BK)t} \de t } \\
    &=\Tr{Q\int_{0}^{\infty} e^{(A-BK)t}e^{(A-BK)^Tt} \de t } \ge \eigmin{Q}\Tr{Y_K}.
\end{split}
\end{align*}
Hence, the right inequality in \eqref{eq:YKBound} follows readily.
\end{proof}

\begin{lemma}\label{lm:PKPoptLowerBound}
    For any $K \in \mathcal{G}$, $\Tr{P_K - P^*} \geq \alpha_4(\norm{K-K^*}_F)$, where $\alpha_4$ is a $\mathcal{K}_\infty$-function defined as
    \begin{align*}
        \alpha_4(r) := \frac{\eigmin{R}r^2}{2\norm{A-BK^*}_F + 2 \norm{B} r}, \quad \forall r \ge 0.
    \end{align*}
\end{lemma}
\begin{proof}

Taking the trace of \eqref{eq:PKdiff1}, and using the cyclic property of the trace in Lemma \ref{lm:traceCyclic}, we have
\begin{align*}
    \Tr{P_K - P^*} &= \Tr{(K-K^*)^T R (K - K^*) Y_K} \\
    &\ge  \eigmin{Y_K} \eigmin{R} \norm{K-K^*}_F^2.
\end{align*}
Considering the lower bound of $\eigmin{Y_K}$ in Lemma \ref{lm:eigminYKLowerBound}, we can obtain 
\begin{align*}
    \Tr{P_K - P^*}  \ge  \frac{\eigmin{R}\norm{K-K^*}_F^2}{2\norm{A-BK^*}_F + 2 \norm{B} \norm{K-K^*}_F} = \alpha_4(\norm{K-K^*}_F).
\end{align*}
\end{proof}

\begin{lemma}[Lemma 3.3 in \cite{bu2020policy}]\label{lm:coercive}
    The objective function $\mathcal{J}_2(K)$ is coercive, i.e. for any sequence $\{K_k\}_{k=0}^\infty$, $K_k \to \partial \mathcal{G}$ or $\norm{K_k}_F \to \infty$, it holds $\mathcal{J}_2(K_k) \to \infty$, as $k \to \infty$. 
\end{lemma}

\begin{lemma}\label{lm:MKlowerbound}
    For any $K \in \mathcal{G}$, let $K' := R^{-1}B^TP_K$ and $M_K := (K - K')^T R (K - K')$. Then,
    \begin{align}\label{eq:KK'geqKdiff}
        \Tr{M_K} \geq a \norm{K - K^*}_F^2 + a' \Tr{P_K-P^*},
    \end{align}
    where $a$ and $a'$ are constants defined as
    \begin{align}\label{eq:aa'def}
        a := \frac{\eigmin{R}\eigmin{Y^*}}{2\eigmin{Y^*}+2\eigmax{Y^*}}, \quad 
        a' := \frac{1}{\eigmin{Y^*} + \eigmax{Y^*}}.
    \end{align}
\end{lemma}
\begin{proof}
    We can rewrite \eqref{eq:continuousLyapunov} as
    \begin{align*}
    \begin{split}
        &(A-BK^*)^TP_K + P_K(A-BK^*) + Q + K^TRK \\
        &+ (K^* - K)^TB^TP_K + P_KB(K^* - K) = 0.  
    \end{split}
    \end{align*}
    Considering $K' = R^{-1}B^T P_K$ and completing the squares, we have
    \begin{align}\label{eq:AoptPK}
    \begin{split}
        &(A-BK^*)^TP_K + P_K(A-BK^*) + Q + (K^*)^T R K^*  \\
        & + (K - K')^T R (K - K')- (K' - K^*)^T R (K' - K^*) = 0.  
    \end{split}
    \end{align}
    Subtracting \eqref{eq:AoptPopt} from \eqref{eq:AoptPK} yields
    \begin{align}\label{eq:AoptPKPoptDiff}
    \begin{split}
        &(A-BK^*)^T(P_K - P^*) + (P_K - P^*)(A-BK^*)  \\
        & + (K - K')^T R (K - K')- (K' - K^*)^T R (K' - K^*) = 0. 
    \end{split}
    \end{align}
    Since $A-BK^*$ is Hurwitz, according to Lemma \ref{lm:LyaEquaIntegral}, we have
    \begin{align}\label{eq:PKPoptYopt}
    \begin{split}
        P_K - P^* &= \int_{0}^{\infty} e^{(A-BK^*)^{T}t} [M_K - (K' - K^*)^T R (K' - K^*) ] e^{(A-BK^*)t} \de t. 
    \end{split}
    \end{align} 
    Taking the trace of \eqref{eq:PKPoptYopt}, and using the cyclic property of the trace in Lemma \ref{lm:traceCyclic} yield
    \begin{align*}
        \Tr{P_K - P^*} 
        = \Tr{Y^* M_K} - \Tr{Y^* (K' - K^*)^TR(K'-K^*)}. 
    \end{align*}     
    By the trace inequality in Lemma \ref{lm:traceIneq}, we have 
    \begin{align}\label{eq:MKupperbound1}
    \begin{split}
         \eigmax{Y^*}\Tr{M_K} &\geq \eigmin{Y^*}\innprod{(K'-K^*)}{R(K'-K^*)} \\
         &+  \Tr{P_K - P^*}.
    \end{split}
    \end{align}

    Using Lemma \ref{lm:CSInequality} and Young's inequality, we can obtain
    \begin{align}\label{eq:MKupperbound2}
    \begin{split}
        &\innprod{K-K^*}{R(K-K^*)} \\
        &= \innprod{(K-K')+(K'-K^*)}{R[(K-K')+(K'-K^*)]} \\
        &= \innprod{K-K'}{R(K-K')} + \innprod{K'-K^*}{R(K'-K^*)} \\
        &+ 2\innprod{K-K'}{R(K'-K^*)} \\
        & \le 2\innprod{K-K'}{R(K-K')} + 2\innprod{K'-K^*}{R(K'-K^*)}.
    \end{split}
    \end{align}
    Noticing that $\Tr{M_K} = \innprod{K-K'}{R(K-K')}$. Plugging \eqref{eq:MKupperbound1} into \eqref{eq:MKupperbound2} and using the trace inequality in Lemma \ref{lm:traceIneq} yield
    \begin{align}\label{eq:MKupperbound3}
    \begin{split}
        &\eigmin{R} \norm{K-K^*}_F^2 \le \left(2 + 2\frac{\eigmax{Y^*}}{\eigmin{Y^*}}\right)\Tr{M_K} \\
        &- \frac{2}{\eigmin{Y^*}}\Tr{P_K - P^*}.
    \end{split}
    \end{align}
    Hence, \eqref{eq:KK'geqKdiff}  follows from \eqref{eq:MKupperbound3}.
\end{proof}

\subsection{Perturbed Standard Gradient Flow}
In this subsection, we apply small-disturbance ISS to analyze the robustness property of the perturbed standard gradient flow for the LQR problem. At any position $K \in \mathcal{G}$, the steepest-descent direction of $\mathcal{J}_2(K)$ is the solution to the problem \cite{book_Nocedal} 
\begin{align*}
    \min_{E} \innprod{E}{\nabla \mathcal{J}_2(K)}, \quad \text{subject to } \innprod{E}{E} = 1.
\end{align*}
Therefore, by the method of Lagrange multipliers, the steepest-descent direction of $\mathcal{J}_2(K)$ is 
\[
-\nabla \mathcal{J}_2(K)/\norm{\nabla \mathcal{J}_2(K)}_F \,.
\]
Consequently, along the the steepest-descent direction and considering the expression of $\nabla \mathcal{J}_2(K)$ in \eqref{eq:JcGradient}, the standard gradient flow is \cite{bu2020policy,Mohammadi2022}
\begin{align*}
    \frac{\de K(s)}{\de{s}} = -\eta \nabla \mathcal{J}_2(K(s)) = -2\eta(RK(s) - B^TP(s))Y(s), 
\end{align*}
where $\eta >0$ is a constant, $P(s) := P_{K(s)}$, and $Y(s) := Y_{K(s)}$. It is seen from \eqref{eq:JcGradient} that the calculation of the gradient $\nabla \mathcal{J}_2(K(s))$ relies on the system matrices, which are unknown in the setting of model-free RL. Many data-driven methods, e.g. approximate dynamic programming \cite{tutorial_Jiang} and finite-difference algorithms \cite{fazel2018global}, are proposed to numerically approximate the gradient. Therefore, in practice, $\widehat{\nabla \mathcal{J}_2}(K(s))$, instead of ${\nabla \mathcal{J}_2}(K(s))$, is utilized to optimize the control gain, and the perturbation $W(s) = \eta[{\nabla \mathcal{J}_2}(K(s)) -\widehat{\nabla \mathcal{J}_2}(K(s))]$ is unavoidable for the gradient flow. Hence, with the perturbation $W \in \mathcal{L}_{\infty}^{m \times n}$, the perturbed gradient flow of the LQR problem is
\begin{align}\label{eq:gradientflow_con}
    \frac{\de K(s)}{\de{s}} =  -2\eta(RK(s) - B^TP(s))Y(s) + W(s).
\end{align}

The following lemma shows that the gradient of $\mathcal{J}_2(K)$ is lower bounded by a $\mathcal{K}$-function of the deviation from the optimal value (gradient dominance condition), which is an important property of a proper objective function in Definition \ref{def:ProperLossFun}.

\begin{lemma} \label{lm:gradientClassK}
 There exists a $\mathcal{K}$-function $\xi_1$, such that for any $K \in \mathcal{G}$, 
\begin{align*}
        \norm{\nabla \mathcal{J}_2(K)}_F \geq  \xi_1(\mathcal{J}_2(K) - \mathcal{J}_2(K^*)).
    \end{align*}
\end{lemma}
\begin{proof}
    Considering the expression of $\nabla \mathcal{J}_2(K)$ in \eqref{eq:JcGradient}, the expression of $K'$ in Lemma \ref{lm:MKlowerbound}, and the cyclic property of the trace in Lemma \ref{lm:traceCyclic}, we have 
    \begin{align*}
         \Tr{\nabla \mathcal{J}_2^T(K)\nabla \mathcal{J}_2(K)} = 4\Tr{Y_K^2(K-K')^TR^2(K-K')}.
    \end{align*}
    By the trace inequality in Lemma \ref{lm:traceIneq}, it holds
    \begin{align*}
    \begin{split}
        &\Tr{\nabla \mathcal{J}_2^T(K) \nabla \mathcal{J}_2(K)} \ge 4 \eigmin{Y_K}^2\Tr{(K-K')^T\sqrt{R}R\sqrt{R}(K-K')} \\
        &= 4 \eigmin{Y_K}^2\Tr{\sqrt{R}(K-K')(K-K')^T\sqrt{R}R} \\
        &\ge 4 \eigmin{Y_K}^2 \eigmin{R}\Tr{\sqrt{R}(K-K')(K-K')^T\sqrt{R}} \\
        &= 4 \eigmin{Y_K}^2 \eigmin{R}\Tr{M_K}.
    \end{split}
    \end{align*}
    By Lemmas \ref{lm:eigminYKLowerBound} and \ref{lm:MKlowerbound}, it follows that
    \begin{align}\label{eq:gradientDominant1}
    \begin{split}
         \Tr{\nabla \mathcal{J}_2^T(K) \nabla \mathcal{J}_2(K)} &\geq \eigmin{R}  \frac{a\norm{K-K^*}_F^2 + a'\Tr{P_K - P^*}} {(\norm{A-BK^*}_F +  \norm{B} \norm{K-K^*}_F)^2}.
    \end{split}
    \end{align}
    To simplify notations, let 
    \begin{align}\label{eq:defa1234}
        a_1 := \eigmin{R}a, \quad a_2 := \eigmin{R}a', \quad a_3 := \norm{A-BK^*}_F, \quad a_4 := \norm{B},
    \end{align}
    and 
    \begin{align*}
         \sigma(r, p) := \frac{a_1 r^2 + a_2 p} {(a_3 +  a_4 r)^2}.
    \end{align*}
    It is clear that $\sigma(r,p) \ge \sigma(r,0)$ for $p \ge 0$. Taking the derivative of $\sigma(r, p)$ with respect to $r$, we have
    \begin{align*}
        \frac{\partial \sigma(r,p)}{\partial r} = \frac{2a_1a_3r  - 2a_2a_4p}{(a_3 +  a_4 r)^3}.
    \end{align*}
    When $p \ge \frac{a_1a_3 }{a_2 a_4} r$, $\sigma(r,p)$ is strictly decreasing in $r$. When $p \le \frac{a_1a_3 }{a_2 a_4} r$, $\sigma(r,p)$ is strictly increasing in $r$. Hence, for each $p$, $r=\frac{a_2a_4}{a_1a_3} p$ is the minimum point of $\sigma(r,p)$, and 
    \begin{align}\label{eq:gradientDominant2}
        \sigma(r,p) \geq \sigma(\frac{a_2a_4}{a_1a_3} p,p) \geq \sigma(\frac{a_2a_4}{a_1a_3} p,0)=: \xi^2_{1}(p),
    \end{align}
    where 
    \begin{align*}
        \xi_{1}(p) := \frac{a_5 p}{a_3 + a_6p},
    \end{align*}
    and 
    \begin{align}\label{eq:a56def}
        a_5 := \frac{a_2a_4}{\sqrt{a_1}a_3}, \quad a_6 := \frac{a_2a_4^2}{a_1a_3}.
    \end{align}
    Since $\frac{\de \xi_1(p)}{\de p} = \frac{a_3 a_5}{(a_3 + a_6p)^2} > 0$, $\xi_1$ is a $\mathcal{K}$-function with the range $[0,\frac{\sqrt{a_1}}{a_4})$.

    Plugging \eqref{eq:gradientDominant2} into \eqref{eq:gradientDominant1} yields
    \begin{align*}
        \norm{\nabla \mathcal{J}_2(K)}_F = \Tr{\nabla \mathcal{J}_2^T(K) \nabla \mathcal{J}_2(K)}^{\frac{1}{2}} \ge  \xi_1\left(\Tr{P_K - P^*}\right).
    \end{align*}    
\end{proof}
\begin{remark}
As remarked in the Introduction, the gradient dominance condition in Lemma \ref{lm:gradientClassK}, the CJS-PL condition, can be considered as a generalization of the well-known Polyak-Łojasiewicz (PL) condition which only holds on a compact set of stabilizing control gains \cite{Mohammadi2022,bu2020policy}. The gradient dominance condition in Lemma \ref{lm:gradientClassK} removes the restriction to compact sets.
\end{remark}

We next revisit the one-dimensional system mentioned in the Introduction, in order to illustrate Lemma \ref{lm:gradientClassK}. Suppose that $m=n=1$ and $A=B=Q=R=1$. In this case, the admissible set is $\mathcal{G}=\{K|K > 1\}$, and one obtains
\begin{align}\label{eq:oneDExample}
\begin{split}
    &P^* = 1+\sqrt{2}, \quad K^* = 1+\sqrt{2}, \quad Y_K = \frac{1}{2(K-1)}, \quad Y^* = \frac{\sqrt{2}}{4} ,   \\
    & \mathcal{J}_2(K) = \frac{1 + K^2}{2(K-1)}, \quad \mathcal{J}_2(K) - \mathcal{J}_2(K^*) = \frac{(K-K^*)^2}{2(K-1)} \,.
\end{split}
\end{align}
The constants in \eqref{eq:aa'def}, \eqref{eq:defa1234} and \eqref{eq:a56def} can be computed as
\begin{align*}
    &a = \frac{1}{4}, \quad a' = \sqrt{2}, \quad a_1 = \frac{1}{4}, \quad a_2 = \sqrt{2}, \\
    &a_3 = \sqrt{2}, \quad a_4 = 1, \quad a_5 = 2, \quad a_6 = 4.
\end{align*}
Consequently, we have
\begin{align*}
    \xi_1(p) = \frac{2p}{\sqrt{2} + 4p}, \quad \xi_1\left(\mathcal{J}_2(K) - \mathcal{J}_2(K^*)\right) = \frac{(K-K^*)^2}{\sqrt{2}(K-1) + 2(K-K^*)^2}.
\end{align*}
The gradient of $\mathcal{J}_2(K)$ is 
\begin{align*}
    \nabla \mathcal{J}_2(K) = \frac{K^2 - 2K - 1}{2(K-1)^2} = \frac{(K-K^*)^2 + 2\sqrt{2}(K-K^*)}{2(K-1)^2}.
\end{align*}
When $K \ge K^*$, $\nabla \mathcal{J}_2(K) \ge 0$, and 
\begin{align*}
    &\frac{\nabla \mathcal{J}_2(K)}{\xi_1\left(\mathcal{J}_2(K) - \mathcal{J}_2(K^*)\right)} \\
    &= \frac{2(K-K^*)^3 + 5\sqrt{2}(K-K^*)^2 + 6(K-K^*) + 4\sqrt{2}}{2(K-K^*)^3 + 4\sqrt{2}(K-K^*)^2 + 4(K-K^*)} \ge 1.
\end{align*}
When $K^* \ge K > 1$,  $\nabla \mathcal{J}_2(K) \le 0$, and 
\begin{align*}
\begin{split}
    &{-\nabla \mathcal{J}_2(K)} - {\xi_1\left(\mathcal{J}_2(K) - \mathcal{J}_2(K^*)\right)} \\
    &= \frac{-4(K-1)^4 + 7\sqrt{2}(K-1)^3 -4(K-1)^2 -6\sqrt{2}(K-1) + 8}{2\sqrt{2}(K-1)^3 + 4 (K-1)^2(K-K^*)^2} \ge 0.
\end{split}
\end{align*}
Therefore, $\norm{\nabla \mathcal{J}_2(K)} \ge \xi_1\left(\mathcal{J}_2(K) - \mathcal{J}_2(K^*)\right)$, which is consistent with Lemma \ref{lm:gradientClassK}.

Based on Lemma \ref{lm:gradientClassK}, we are ready to state the main result on the small-disturbance ISS property of the perturbed standard gradient flow in \eqref{eq:gradientflow_con}.

\begin{theorem}\label{thm:smallISSGradientFlow}
    System \eqref{eq:gradientflow_con} is small-disturbance ISS with respect to $W$. 
\end{theorem} 
\begin{proof}
Define 
\begin{align}\label{eq:costV3}
    \mathcal{V}_3(K) := \mathcal{J}_2(K) - \mathcal{J}_2(K^*)
\end{align}
It follows from \eqref{eq:costJc_closedform} that $\mathcal{V}_3(K) = \Tr{P_K - P^*}$. Clearly, $\mathcal{V}_3(K)$ is smooth in $K$ \cite[Proposition 3.2]{bu2020policy}. Since $K^*$ is the unique minimum of $\mathcal{J}_2(K)$, $\mathcal{V}_3(K)$ is a positive definite function with respect to $K^*$. The coercivity of $\mathcal{V}_3(K)$ can be obtained by Lemma \ref{lm:coercive}. Therefore, by Definition \ref{def:sizeFunc}, $\mathcal{V}_3(K)$ is a size function for $(\mathcal{G},K^*)$. Hence, by Lemma \ref{lm:gradientClassK} and Definition \ref{def:ProperLossFun}, $\mathcal{J}_2$ is a proper objective function. According to Theorem \ref{eq:gradFlowISS}, the proof of Theorem \ref{thm:smallISSGradientFlow} is completed.
\end{proof}

As a direct consequence of Theorem \ref{thm:smallISSGradientFlow}, an estimate of $\norm{K(s) - K^*}_F$ is provided in the following corollary.
\begin{corollary}
    There exist a constant $d_1 > 0$, a $\mathcal{KL}$-function $\beta_2(\cdot,\cdot)$, and a $\mathcal{K}_{[0,d_1)}$-function $\gamma_4$, such that for all perturbations $W$ essentially bounded by $d_1$ (i.e. $\norm{W}_\infty < d_1$), and all initial conditions $K(0) \in \mathcal{G}$, $K(s)$ satisfies
    \begin{align}\label{eq:KsmallISS}
        \norm{K(s) - K^*}_F \le \beta_2(\mathcal{V}_3(K(0)), s) + \gamma_4(\norm{W}_\infty), \quad \forall s\ge 0,
    \end{align}
where $\mathcal{V}_3$ is defined in \eqref{eq:costV3}.
\end{corollary}
\begin{proof}
    Since system \eqref{eq:gradientflow_con} is small-disturbance ISS, when $\norm{W}_\infty < d_1$, it holds 
    \begin{align*}
        \mathcal{V}_3(K(s)) \le \beta_2'(\mathcal{V}_3(K(0)), s) + \gamma_4'(\norm{W}_\infty), \quad \forall s \ge 0,
    \end{align*}
    where $\beta_2'$ is a $\mathcal{KL}$-function and $\gamma_4'$ is a $\mathcal{K}_{[0,d_1)}$-function. According to Lemma \ref{lm:PKPoptLowerBound} and recalling that $\mathcal{V}_3(K(s)) = \Tr{P(s) - P^*}$, we have
    \begin{align*}
        \norm{K(s) - K^*}_F \le \alpha_4^{-1}\left(\beta_2'(\mathcal{V}_3(K(0)), s) + \gamma_4'(\norm{W}_\infty)\right).
    \end{align*}
    Using Lemma \ref{lm:weakTriangle}, we can obtain \eqref{eq:KsmallISS}.
\end{proof}

\subsection{Perturbed Natural Gradient Flow}
In this subsection, we will show that the perturbed natural gradient flow is small-disturbance ISS. It follows from \eqref{eq:TrPKPopt} that
\begin{align*}
    \mathcal{J}_2(K) - \mathcal{J}_2(K^*) = \innprod{K-K^*}{R(K-K^*)}_{Y_K}.
\end{align*}
Hence, the objective function can be viewed as a quadratic function over the Riemannian manifold $(\mathcal{G}, \innprod{\cdot}{\cdot}_{Y_K})$. As seen in the expression of $\nabla \mathcal{J}_2(K)$, the magnitude of the gradient dependents on $Y_K$, and $Y_K$ may tend to infinity when $K \to \partial \mathcal{G}$, and tend to zero when $\norm{K}_F \to \infty$ (see the illustrative one-dimensional system in \eqref{eq:oneDExample}). The non-isotropic property of the magnitude of the gradient may slow down the convergence of the gradient flow. To handle the non-isotropic property, in \cite{Amari1998Why,Amari1998}, the natural gradient was proposed as a way to modify the standard gradient search direction according to the Riemannian structure of the parameter space.   

Over the Riemannian manifold $(\mathcal{G}, \innprod{\cdot}{\cdot}_{Y_K})$, the steepest-descent direction can be obtained by solving
\begin{align}\label{eq:Naturalsteepest-descent}
    \min_{E} \innprod{E}{\nabla \mathcal{J}_2(K)}, \quad \text{subject to } \innprod{E}{E}_{Y_K} = 1.
\end{align}
By the method of Lagrange multipliers, the solution of \eqref{eq:Naturalsteepest-descent} is 
\begin{align*}
    E = -\nabla \mathcal{J}_2(K)Y_K^{-1}/ \innprod{\nabla \mathcal{J}_2(K)Y_K^{-1}}{\nabla \mathcal{J}_2(K)Y_K^{-1}}_{Y_K}^{1/2}.
\end{align*}
The natural gradient of $\mathcal{J}_2(K)$ over the Riemannian manifold $(\mathcal{G}, \innprod{\cdot}{\cdot}_{Y_K})$, denoted by $\grad{\mathcal{J}_2(K)}$, is 
\begin{align*}
    \grad{\mathcal{J}_2(K)} := \nabla \mathcal{J}_2(K) Y_K^{-1} = 2(RK - B^TP_K).
\end{align*}
Considering the perturbation, the natural gradient flow is
\begin{align}\label{eq:GradientFlow_NaturalCont}
    \frac{\de K(s)}{\de s} = - 2\eta (RK(s) - B^TP(s)) + W(s).
\end{align}

The following lemma is introduced to pave the foundation for the proof of small-disturbance ISS property of system \eqref{eq:GradientFlow_NaturalCont}.
\begin{lemma}\label{lm:PdiffNatural}
    For any $K\in\mathcal{G}$, we have
    \begin{align}\label{eq:KKoptK'Geo}
        2\innprod{K-K^*}{R(K-K')}_{Y^*} = \Tr{P_K - P^*} + \innprod{K-K^*}{R(K-K^*)}_{Y^*}.
    \end{align}
\end{lemma}
\begin{proof}
    By completing the squares, we have
    \begin{align*}
     \begin{split}
         &(K - K')^T R (K - K')  - (K' - K^*)^T R (K' - K^*)  \\
         &= (K - K')^T R (K - K^*) + (K - K^*)^T R (K - K') - (K - K^*)^T R (K - K^*).  
    \end{split}   
    \end{align*}    
    Therefore, we can rewrite \eqref{eq:AoptPKPoptDiff} as
    \begin{align*}
    \begin{split}
        &(A-BK^*)^T(P_K - P^*) + (P_K - P^*)(A-BK^*) + (K - K')^T R (K - K^*) \\
        &+(K - K^*)^T R (K - K') - (K - K^*)^T R (K - K^*) = 0. 
    \end{split}
    \end{align*}
    Since $A-BK^*$ is Hurwitz, by Lemma \ref{lm:LyaEquaIntegral} we have
    \begin{align}\label{eq:PKdiffRewrite}
    \begin{split}
        &P_K - P^* = \int_{0}^{\infty} e^{(A-BK^*)^{T}t } [(K - K')^T R (K - K^*) \\
        &+(K - K^*)^T R (K - K') - (K - K^*)^T R (K - K^*)]e^{(A-BK^*)t } \de t. 
    \end{split}
    \end{align}
    Taking the trace of \eqref{eq:PKdiffRewrite} and using the cyclic property of trace in Lemma \ref{lm:traceCyclic}, we have \eqref{eq:KKoptK'Geo}.
\end{proof}

Next, we will prove the small-disturbance ISS property of the perturbed natural gradient flow.

\begin{theorem} \label{thm:ISSnature}
Given $Q \succ 0$, system  \eqref{eq:GradientFlow_NaturalCont} is small-disturbance ISS with respect to $W$.
\end{theorem}
\begin{proof}
Let $\mathcal{V}_4(K) := \frac{1}{2} \innprod{K - K^*}{K - K^*}_{Y^*}$. Clearly, by Lemma \ref{lm:traceIneq}, $\mathcal{V}_4(K)$ is bounded by
\begin{align*}
    \frac{1}{2}\eigmin{Y^*} \norm{K - K^*}_F^2 \leq \mathcal{V}_4(K) \leq \frac{1}{2}\norm{Y^*} \norm{K - K^*}_F^2.
\end{align*}
Then, it holds: 
\begin{align*}
\begin{split}
    &\innprod{\nabla \mathcal{V}_4(K)}{- 2\eta (RK - B^TP_K) + W} \\
    &= -2\eta \innprod{K - K^*}{R(K - K')}_{Y^*} + \innprod{K - K^*}{W}_{Y^*}.
\end{split}
\end{align*}
Recall that $K' = R^{-1}B^T P_K$. According to Lemma \ref{lm:PdiffNatural}, we obtain
\begin{align}\label{eq:Vnnature2}
\begin{split}
    &\innprod{\nabla \mathcal{V}_4(K)}{- 2\eta (RK - B^TP_K) + W} \\
    &= - \eta\Tr{P_K - P^*} - \eta\innprod{K - K^*}{R(K - K^*)}_{Y^*} + \innprod{K - K^*}{W}_{Y^*}\\
    &\le - \eta\Tr{P_K - P^*} - \eta\innprod{K - K^*}{R(K - K^*)}_{Y^*}  \\
    &+ \innprod{K - K^*}{ K - K^*}_{Y^*}^{\frac{1}{2}}\innprod{W}{ W}_{Y^*}^{\frac{1}{2}}\\
    &\le - \eta\Tr{P_K - P^*} - \frac{\eta\eigmin{R}}{2}\innprod{K - K^*}{K - K^*}_{Y^*} + \frac{\norm{Y^*}}{2\eta\eigmin{R}} \norm{W}_F^2,
\end{split}
\end{align}
where the first inequality follows from Lemma \ref{lm:CSInequality}, and the last inequality is according to Young's inequality and the trace inequality in Lemma \ref{lm:traceIneq}.

Differentiating $\mathcal{V}_3(K) = \mathcal{J}_2(K) - \mathcal{J}_2(K^*)$ and considering \eqref{eq:JcGradient} yield
\begin{align*}
\begin{split}
    &\innprod{\nabla \mathcal{V}_3(K)}{- 2\eta (RK - B^TP_K) + W} 
    = -4 \eta \innprod{RK - B^TP_K}{ RK - B^TP_K}_{Y_K} \\
    &+ 2\innprod{RK - B^TP_K}{W}_{Y_K} \\
    &\le -3 \eta \innprod{RK - B^TP_K}{ RK - B^TP_K}_{Y_K} + \frac{1}{\eta} \innprod{W}{ W}_{Y_K},
\end{split}
\end{align*} 
where the last line is obtained by Lemma \ref{lm:CSInequality} and Young's inequality. Using Lemma \ref{lm:YKBound} and the trace inequality in Lemma \ref{lm:traceIneq}, we have
\begin{align}\label{eq:VvNaturalFlow2}
    &\innprod{\nabla \mathcal{V}_3(K)}{- 2\eta (RK - B^TP_K) + W}  \le -3 \eta \innprod{RK - B^TP_K}{ RK - B^TP_K}_{Y_K} \nonumber\\
    &+ \frac{\mathcal{V}_3(K) + \Tr{P^*}}{\eta \eigmin{Q}} \norm{W}_F^2.
\end{align}

Let $\mathcal{V}_5(K) := \mathcal{V}_3(K) + \mathcal{V}_4(K)$. Since $\mathcal{V}_3$ is a size function and $\mathcal{V}_4$ is positive definite with respect to $K^*$, $\mathcal{V}_5$ is also a size function. It follows from \eqref{eq:Vnnature2} and \eqref{eq:VvNaturalFlow2} that 
\begin{align}\label{eq:Vcnaturalflow1}
\begin{split}
    \innprod{\nabla \mathcal{V}_5(K)}{- 2\eta (RK - B^TP_K) + W} &\le -\eta \mathcal{V}_3(K) - {\eta \eigmin{R}} \mathcal{V}_4(K) \\
    &+ \frac{\mathcal{V}_3(K) + b_1}{\eta b_2} \norm{W}_F^2,
\end{split}    
\end{align}
where
\begin{align*}
    b_1 := \frac{\norm{Y^*}\eigmin{Q}}{2\eigmin{R}} + \Tr{P^*}, \, b_2 := \eigmin{Q}, \, b_3 := \eta^2\eigmin{R}\eigmin{Q}.
\end{align*}
Without losing generality, assume that $\eigmin{R}\le1$. Then, it follows from \eqref{eq:Vcnaturalflow1} that
\begin{align*}
\begin{split}
    \innprod{\nabla \mathcal{V}_5(K)}{- 2\eta (RK - B^TP_K) + W} &\le -\eta\eigmin{R} \mathcal{V}_5(K)  + \frac{\mathcal{V}_5(K) + b_1}{\eta b_2} \norm{W}_F^2.
\end{split}        
\end{align*}
Thus, if 
\begin{align*}
    \norm{W}_F \le \left(\frac{b_3 \mathcal{V}_5(K)}{2\mathcal{V}_5(K)+2b_1} \right)^{\frac{1}{2}} =: \xi_2(\mathcal{V}_5(K)), 
\end{align*}
it is guaranteed that
\begin{align*}
    \innprod{\nabla \mathcal{V}_5(K)}{- 2\eta (RK - B^TP_K) + W} &\le -\frac{\eta\eigmin{R}}{2} \mathcal{V}_5(K).
\end{align*}
Since $
\frac{\de \xi_2(r)}{\de r} = \frac{1}{2}\xi_2^{-1}(r) \frac{ 2b_1 b_3}{(2r+2b_1)^2} > 0, \quad \forall r>0$, $\xi_2$ is a $\mathcal{K}$-function, and its range is $[0, \sqrt{\frac{b_3}{2}})$. Consequently, $\mathcal{V}_5$ is a small-disturbance ISS-Lyapunov function. According to Theorem \ref{thm:smallISSSufficient}, we conclude that system \eqref{eq:GradientFlow_NaturalCont} is small-disturbance ISS. 
\end{proof}

\subsection{Perturbed Newton Gradient Flow}
By applying the results of Section 3, we will show that the perturbed Newton gradient flow is small-disturbance ISS. The Newton gradient descent method was first adopted in \cite{Kleinman1968} for solving the LQR problem, and it converges to the optimum at a quadratic convergence rate. The Newton direction is derived from the second-order Taylor series approximation of $\mathcal{J}_2(K+E)$, which, according to Lemma \ref{lm:J2Expansion}, can be expressed as
\begin{align*}
    \mathcal{J}_2(K+E) = \mathcal{J}_2(K) + \innprod{E}{2(RK-B^TP_K)}_{Y_K} + \innprod{E}{RE}_{Y_K} + O(\norm{E}_F^2).
\end{align*}
By minimizing the second-order Taylor series approximation of $\mathcal{J}_2(K)$ over $E$, the Newton direction is obtained as $-(K - R^{-1}B^T P_K)$. 
 Considering the perturbation, the Newton gradient flow is
\begin{align}\label{eq:GradientFlow_Newton}
    \frac{\de K(s)}{\de s} = -\eta(K(s) - R^{-1}B^TP(s)) + W(s).
\end{align}
\begin{theorem}  \label{thm:smallISSNewton}
Given $Q \succ 0$, system \eqref{eq:GradientFlow_Newton} is small-disturbance ISS with respect to $W$.
\end{theorem}
\begin{proof}
The proof follows from the proof of Theorem \ref{thm:ISSnature} by defining $\mathcal{V}_6(K) = \mathcal{V}_3(K) + \frac{1}{2} \innprod{K - K^*}{R(K - K^*)}_{Y^*}$.
\end{proof}

\section{Conclusions}
In this paper, we studied the small-disturbance ISS property of continuous-time gradient flows on an open subset of certain Euclidean space. In the framework of small-disturbance ISS, the transient behavior, the convergence speed, and the robustness to the perturbations of gradient flows can be well quantified. As a by-product, a Lyapunov characterization of small-disturbance ISS is given. Upon specification to the policy optimization of the LQR problem, three kinds of perturbed gradient flows, including standard gradient flow, natural gradient flow, and Newton gradient flow, were studied in greater details. In particular, they are all small-disturbance ISS.

\bibliographystyle{alpha}
\bibliography{reference}

\end{document}

%% file: packages.tex
\newcommand{\LC}[1]{\textcolor{black}{#1}}

\newcommand{\R}{{\mathbb R}}  
\newcommand{\Rn}{\R^{n}}
\newcommand{\st}{\,|\,}
\newcommand{\abs}[1]{\left|#1\right|}

\usepackage{graphicx}
\usepackage[colorlinks=true, allcolors=blue]{hyperref}
\usepackage{amsmath}
\usepackage{amsfonts}
\usepackage{amsthm}
\usepackage{amssymb}
\usepackage{subcaption}

\newtheorem{theorem}{Theorem}[section]

\newtheorem{lemma}{Lemma}[section]

\newtheorem{remark}{Remark}[section]
\newtheorem{definition}{Definition}[section]
\newtheorem{corollary}{Corollary}[section]

\newcommand\norm[1]{\lVert#1\rVert}
\newcommand\Tr[1]{\mathrm{Tr}\left(#1\right)}
\newcommand\grad[1]{\mathrm{grad}\left(#1\right)}
\newcommand\eigmax[1]{\lambda_{\mathrm{max}}\left(#1\right)}
\newcommand\eigmin[1]{\lambda_{\mathrm{min}}\left(#1\right)}
\newcommand{\de}{\mathrm{d}}
\newcommand\innprod[2]{\langle #1, #2 \rangle}
\DeclareMathOperator*{\esssup}{ess\,sup}